\documentclass[11pt]{amsart}

\usepackage{latexsym}
\usepackage{amssymb}
\usepackage{amsmath}

\newtheorem{theorem}{Theorem}[section]
\newtheorem{lemma}[theorem]{Lemma}
\newtheorem{proposition}[theorem]{Proposition}
\newtheorem{corollary}[theorem]{Corollary}

\theoremstyle{definition}

\newtheorem{example}[theorem]{Example}

\newtheorem{remark}[theorem]{Remark}

\def\N{\mathbb{N}}
\def\Z{\mathbb{Z}}
\def\R{\mathbb{R}}

\def\dd{{\mathfrak{d}}}

\begin{document}

\title{Intersections of sets of distances}

\author{Mauro Di Nasso}

\address{Dipartimento di Matematica\\
Universit\`a di Pisa, Italy}

\email{dinasso@dm.unipi.it}

\date{}

\begin{abstract}
We isolate conditions on the relative asymptotic size 
of sets of natural numbers
$A,B$ that guarantee a nonempty intersection 
of the corresponding sets of distances. Such conditions
apply to a large class of zero density sets.
We also show that a variant of \emph{Khintchine's Recurrence Theorem}
holds for all infinite sets $A=\{a_1<a_2<\ldots\}$ where $a_n\lll n^{3/2}$.
\end{abstract}



\maketitle

\section{Introduction}

It is a well-known phenomenon that if a set of natural numbers
$A$ has positive upper asymptotic density $\overline{d}(A)>0$, 
then its \emph{set of distances} (or \emph{Delta-set})
$$\Delta(A)\ =\ \{a-a'\mid a,a'\in A,\ a>a'\}$$
has a really rich combinatorial structure.
An old problem attributed to Paul Erd\"os
was whether the distance sets of two sets of
positive upper density must necessarily meet.

\smallskip
\begin{itemize}
\item
\emph{Does $\Delta(A)\cap\Delta(B)\ne\emptyset$ whenever
$\overline{d}(A),\overline{d}(B)>0$ ?}
\end{itemize}

The answer was shortly shown to be positive, and in fact
the following much stronger intersection property holds:

\smallskip
\begin{itemize}
\item
\emph{If the upper density $\overline{d}(A)=\alpha>0$ is positive, then
$\Delta(A)\cap\Delta(B)\ne\emptyset$ whenever the
set $B$ contains more than $1/\alpha$-many elements.}
\end{itemize}

The proof consists of a straightforward application of the
\emph{pigeonhole principle}.
The key observation is that if one takes distinct elements
$b_1,\ldots,b_N$ with $N>1/\alpha$, then
the shifted sets $A+b_i$ cannot be pairwise disjoint, as otherwise
$\overline{d}(\bigcup_{i=1}^N A+b_i)=\sum_{i=1}^N\overline{d}(A+b_i)=
N\cdot\overline{d}(A)>1$.
The argument is then completed by noticing that 
$(A+b)\cap(A+b')\ne\emptyset$ for some $b\ne b'$ in $B$
if and only if $\Delta(A)\cap\Delta(B)\ne\emptyset$.

In the last forty years, the research on the combinatorial properties of 
distance sets and difference sets\footnote
{~By a \emph{difference set} is meant a set of the
form $A-B=\{a-b\mid a\in A, b\in B\}$. So, the 
set of distances $\Delta(A)$ is the positive part
of $A-A$.} has produced many interesting 
results (see, \emph{e.g.}, 
\cite{ru1,es,ru2,sa1,sa2,sa3,st,be1,pss,behl,jin,rs,lm,dn})
which are almost always 
grounded on the hypothesis of positive density.
In this paper, we look for general properties that
include the zero density case, and investigate the size 
of intersections $\Delta(A)\cap\Delta(B)$ depending on
the relative density of $A$ with respect to $B$.
More generally, for $k\in\N$, we will consider intersections
$R_k(A)\cap\Delta(B)$ where 
$$R_k(A)\ =\ \{x\in\N : |A\cap (A+x)|\ge k\}$$
is the $k$-\emph{recursion set} of $A$.
Elements of $R_k(A)$ are those natural numbers
that are the distance of at least $k$-many different
pairs of elements in $A$; in particular, $R_1(A)=\Delta(A)$.

\smallskip
The main results presented here (see Corollaries \ref{main1} and \ref{main2})
can be summarized as follows.

\medskip
\noindent
\textbf{Main Theorem.}
\emph{
Let $A=\{a_n\}$ and $B=\{b_n\}$ 
be infinite sets of natural numbers, and let $\vartheta:\N\to\R^+$
be such that 
$\limsup_{n\to\infty}\frac{a_n}{n\cdot\vartheta(n)}<\infty$.}

\smallskip
\begin{enumerate}
\item
\emph{If $\lim_{n\to\infty}\frac{\vartheta(b_n)}{n}=0$, then 
$R_k(A)\cap\Delta(B)$ is infinite for all $k$.}

\smallskip
\item
\emph{If $\lim_{n\to\infty}\frac{\vartheta(n\cdot b_n)}{n}=0$,
then there exists a sequence $\langle x_n\mid n\in\N\rangle$ of 
elements of $\Delta(B)$ such that
$$\limsup_{n\to\infty}\,
\left(\frac
{\frac{|A\cap(A+x_n)\cap[1,n]|}{n}}
{\left(\frac{|A\cap[1,n]|}{n}\right)^2}\right)\ \ge\ 1\,.$$}
\end{enumerate}

\smallskip
We remark that the above results apply to a large class
of zero density sets; \emph{e.g.}, when $B=\N$,
(1) applies whenever $a_n\lll n^2$, and (2) applies whenever 
$a_n\lll n^{3/2}$.
By way of examples, we list below three consequences
(see Examples \ref{ex1}, \ref{ex2}, and \ref{ex3}).

\medskip
\noindent
\textbf{Example 1.}\ 
If $\sum_{n=1}^\infty \frac{1}{a_n}=\infty$ and $B=\{b_n\}$
is such that $\log b_n\lll n^{1-\varepsilon}$ for some $\varepsilon>0$,
then the intersections $R_k(A)\cap\Delta(B)$ are infinite for all $k$.

\medskip
\noindent
\textbf{Example 2.}\ 
Let $A=\{\lfloor K\cdot n\sqrt{n}\rfloor\}$ and $B=\{M\cdot n^3\}$.
If $K^2\cdot M<\frac{4}{27}$ then $R_k(A)\cap\Delta(B)\ne\emptyset$
for all $k$.

\medskip
\noindent
\textbf{Example 3.}\ 
Let $A=\{a_n\}$ have the same asymptotic size as the set of prime numbers,
\emph{i.e.} $\lim_{n\to\infty}\frac{a_n}{n\cdot\log n}=1$,
and assume that $B=\{b_n\}$ is sub-exponential, \emph{i.e.}
$\log b_n\lll n$.
Then for every $\varepsilon>0$ there exist infinitely many $n$
and elements $x_n\in\Delta(B)$ such that
$$|A\cap(A+x_n)\cap[1,a_n]|\ \ge\ \frac{n}{\log n}\cdot (1-\varepsilon).$$

Of course, the asymptotic conditions considered
in our theorems about sequences $A=\{a_n\}$
may be reformulated by using the corresponding
counting functions $A(u)=|\{a\in A\mid a\le u\}|$. 
For instance, the role of ${a_n}/n$ is played
by $u/A(u)$, and so forth.

\medskip
\textbf{Notation.}\
The natural numbers $\N$ are 
the set of \emph{positive} integers. Letters 
$n,m,h,k,s,t,\nu,\mu,N$ will be used for natural numbers, 
and upper-case letters $A,B,C,$
will be used for sets of natural numbers. 
For infinite sets $A\subseteq\N$ we use the brace notation 
$A=\{a_n\}$ to mean that elements $a_n$ are
arranged in increasing order:
$$A\ =\ \{a_n\}\ =\ \{a_1<a_2<\ldots<a_n<\ldots\}.$$
We write $A+x=\{a+x\mid a\in A\}$ to denote the \emph{shift} of $A$ by $x$.
For functions $f:\N\to\R^+$ taking 
positive real values, we write $a_n\lll f(n)$  to mean that 
$\lim_{n\to\infty}{a_n}/f(n)=0$.\footnote
{~This is equivalent to Landau notation $a_n=o(f(n))$.}
By $\lfloor x\rfloor=\max\{k\in\Z\mid k\le x\}$ is denoted
the \emph{integer part} of a real number $x$.
Finally, recall the notion of \emph{upper asymptotic density} $\overline{d}(A)$
for sets $A\subseteq\N$:
$$\overline{d}(A)\ =\ \limsup_{n\to\infty}\frac{|A\cap[1,n]|}{n}.$$

\medskip
\section{Preliminary results}

Let us start with a straightforward consequence of the 
\emph{pigeonhole principle}.

\smallskip
\begin{proposition}\label{prop1}
Let $A=\{a_n\}$  and $B=\{b_n\}$ be infinite sets of natural numbers.
If there exist $N,\nu$ such that $a_N+b_\nu\le N\cdot\nu$
then $\Delta(A)\cap\Delta(B)\ne\emptyset$.
In particular, if $\liminf_{n\to\infty}\frac{a_n+b_n}{n^2}<1$
then $\Delta(A)\cap\Delta(B)\ne \emptyset$.
\end{proposition}

\begin{proof}
Fix $N,\nu$ as in the hypothesis. The sumset 
$$\{a_i+b_j\mid 1\le i\le N;\ 1\le j\le \nu\}\subseteq
[2,a_N+b_\nu]\subseteq[2,N\cdot\nu]$$
contains at most $N\cdot\nu-1$
elements. So, by the \emph{pigeonhole principle},
there exist $(i,j)\ne(i',j')$ such that $a_i+b_j=a_{i'}+b_{j'}$.
Clearly $i\ne i'$, say $i>i'$.
Then $a_i-a_{i'}=b_{j'}-b_j\in\Delta(A)\cap\Delta(B)\ne\emptyset$.
If $\liminf_n\frac{a_n+b_n}{n^2}<1$, pick $N$ such that 
$\frac{a_N+b_N}{N^2}<1$, and apply
the above argument with $N=\nu$.
\end{proof}

\smallskip
\begin{remark}
The above result 
is best possible because there exist infinite sets $A=\{a_n\}$ and $B=\{b_n\}$
such that $\liminf_n\frac{a_n+b_n}{n^2}=1$ but $\Delta(A)\cap\Delta(B)=\emptyset$.
The following example is due to P. Erd\"os and R. Freud  \cite{ef}.

\begin{itemize}
\item 
Let $A$ be the set of all natural numbers that are sums of even powers of $2$,
including $1=2^0$.
\item
Let $B$ be the set of all natural numbers that are sums of odd powers of $2$.
\end{itemize}

It only takes a little computation to verify that:

\begin{itemize}
\item
$b_n=2\cdot a_n$ for all $n$\,;
\item
$\liminf_{n\to\infty}a_n/n^2=1/3$ is attained on the
subsequence $n_k=2^k-1$\,;
\item
$\liminf_{n\to\infty}\frac{a_n+b_n}{n^2}=1$.
\end{itemize}

Besides, an equality $a_i-a_j=b_s-b_t\Leftrightarrow a_i+b_t=a_j+b_s$
holds if and only if $i=j$ and $s=t$, since every natural number
is uniquely written as a sum of powers of $2$. It follows
that $\Delta(A)\cap\Delta(B)=\emptyset$.
\end{remark}

In order to improve on the previous result, we
will use the following elementary inequality.

\smallskip
\begin{lemma}\label{lemma}
Let $A=\{a_1<\ldots<a_N\}$ and $B=\{b_1<\ldots<b_\nu\}$ be 
finite sets of natural numbers. For every $h\le\nu/2$ there exists 
$x\in\Delta(B)$ such that $x\ge h$ and 
$$|\,A\cap (A+x)\,|\ \ge\ 
\frac{N^2}{a_N+b_\nu}-\frac{N\cdot(2h-1)}{\nu}\ =\  
\frac{N^2}{a_N}\cdot
\frac{1-\frac{(a_N+b_\nu)(2h-1)}{N\cdot\nu}}{1+\frac{b_\nu}{a_N}}\,.$$
The above inequality is strict except when $h=1$ and $N\cdot\nu=a_N+ b_\nu$.\end{lemma}

\begin{proof}
Let us first consider the case $h=1$.
Let $I$ be the interval $[1,a_N+b_\nu]$, and for every
$i=1,\ldots,\nu$, let $\chi_i:I\to\{0,1\}$ be the characteristic
function of the shifted sets $A+b_i\subseteq I$. Notice that
$$\sum_{x\in I}\left(\sum_{i=1}^\nu \chi_i(x)\right)\ =\ 
\sum_{i=1}^\nu\left(\sum_{x\in I}\chi_i(x)\right)\ =\ 
\sum_{i=1}^\nu|A+b_i|\ =\ N\cdot\nu.$$
By \emph{Cauchy-Schwartz inequality}, we obtain:
\begin{eqnarray}
\nonumber
N^2\cdot\nu^2 & = & \left(\,\sum_{x\in I}\ 1\cdot\left(\sum_{i=1}^\nu\chi_i(x)\right)\right)^2\ \le \
\left(\sum_{x\in I}\ 1^2\right)\,\cdot\,\sum_{x\in I}\left(\sum_{i=1}^\nu\chi_i(x)\right)^2
\\
\nonumber
{} & = & |I|\cdot\sum_{x\in I}\left(\sum_{i,j=1}^\nu\chi_i(x)\cdot\chi_j(x)\right)\ =\
|I|\cdot\sum_{i,j=1}^\nu\left(\sum_{x\in I}\chi_i(x)\cdot\chi_j(x)\right)
\\
\nonumber
{} & = & (a_N+b_\nu)\,\cdot\, \sum_{i,j=1}^\nu|(A+b_i)\cap(A+b_j)|.
\end{eqnarray}
If $M=\max\{|(A+b_i)\cap(A+b_j)|\,:\,1\le i<j\le\nu\}$ then
\begin{eqnarray}
\nonumber
\sum_{i,j=1}^\nu|(A+b_i)\cap(A+b_j)| & = &
\sum_{i=1}^\nu|A+b_i|+ 2\cdot\!\!\!\!\!\sum_{1\le i<j\le\nu}\!\!\!|(A+b_i)\cap(A+b_j)|
\\
\nonumber
{} & \le & N\cdot\nu + 2\cdot\binom{\nu}{2}\cdot M\  = \
\nu\cdot\left(N + (\nu-1)\cdot M\right)
\end{eqnarray}
(The above expressions are well-defined because  
we are assuming $h\le\nu/2$, and hence $\nu\ge 2$.)
By combining with the previous inequalities, we get that
$$N^2\cdot\nu\ \le\ (a_N+b_\nu)\cdot\left(N + (\nu-1)\cdot M\right),$$
and hence
$$M\ \ge\ \frac{\nu}{\nu-1}\cdot\frac{N^2}{a_N+b_\nu}-\frac{N}{\nu-1}\ =\
\frac{\nu}{\nu-1}\cdot\left(\frac{N^2}{a_N+b_\nu}-\frac{N}{\nu}\right)\ \ge\ 
\frac{N^2}{a_N+b_\nu}-\frac{N}{\nu}.$$
Notice that the last inequality is strict
provided that $\frac{N^2}{a_N+b_\nu}-\frac{N}{\nu}>0$ or, equivalently,
when $N\cdot\nu>a_N+b_\nu$. Notice also that, since $M\ge 0$, the strict inequality
trivially holds also when $N\cdot\nu<a_N+b_\nu$.
Observe that if $M=|(A+b_s)\cap(A+b_t)|$, then $M=|A\cap(A+x)|$ 
where $x=b_s-b_t\in\Delta(B)$, and this completes the proof of the case $h=1$.

Now let $h\ge 2$.
Let $\mu$ be such that $\mu h\le\nu<(\mu+1)h$, and
consider the set $B'=\{b'_1<\ldots<b'_\mu\}\subset B$
where $b'_i=b_{ih}$. Notice that $\mu\ge 2$, because 
we are assuming $h\le\nu/2$, and so we can apply
the property proved above to prove the existence of an
element $x\in\Delta(B')$ such that
$$|\,A\cap (A+x)\,|\ \ge\ 
\frac{N^2}{a_N+b'_\mu}-\frac{N}{\mu}\ \ge\ 
\frac{N^2}{a_N+b_\nu}-\frac{N}{\mu}.$$
For suitable indexes $1\le s<t\le\mu$, 
one has that $x=b'_t-b'_s=b_{th}-b_{sh}\ge th-ts\ge h$.
Finally, notice that $\frac{N}{\mu}=\frac{N}{\nu}\cdot\frac{\nu}{\mu}<
\frac{N}{\nu}\cdot\frac{(\mu+1)h}{\mu}$, and since
$\frac{(\mu+1)h}{\mu}\le 2h-1$, the thesis follows.
Indeed, $\mu h+h\le 2\mu h -\mu\Leftrightarrow \mu h\ge \mu+h$,
and the last inequality holds because $\mu,h\ge 2$.
\end{proof}

\smallskip
\begin{theorem}\label{thm1}
Let $A=\{a_n\}$  and $B=\{b_n\}$ be infinite sets of natural numbers
such that
$$\liminf_{n\to\infty}\frac{a_n+b_n}{n^2}\ =\ 0\,.$$
Then $R_k(A)\cap\Delta(B)$ is infinite for all $k$.\footnote
{~Recall that $R_k(A)=\{x\in\N : |A\cap (A+x)|\ge k\}$.}
\end{theorem}

\begin{proof}
Fix an arbitrary $h\in\N$. 
For every $n\ge 2h$, apply Lemma \ref{lemma} to the finite sets
$A_n=\{a_1<\ldots<a_n\}$ and $B_n=\{b_1<\ldots<b_n\}$,
and get the existence of an element $x_n\in\Delta(B_n)$ such that
$x_n\ge h$ and
$$|A\cap(A+x_n)\cap[1,a_n+b_n]|\ \ge\ 
|A_n\cap (A_n+x_n)|\ >\  
\frac{n^2}{a_n+b_n}-(2h-1)\,.$$

By the hypothesis, the sequence on the right side is 
unbounded as $n$ goes to infinity and so, for every $k$,
there exists $x_n\in\Delta(B_n)\subseteq\Delta(B)$ with 
$x_n\ge h$ and $|A\cap(A+x_n)|\ge k$.
As $h$ was arbitrary, this proves that intersections $R_k(A)\cap\Delta(B)$ 
are infinite.
\end{proof}

\smallskip
Next, we prove that when $A$ has positive asymptotic density, 
the set of all possible shifts $x$ that yield ``large'' 
intersections $A\cap(A+x)$ is ``combinatorially large",
in the sense that it meets all sufficiently large Delta-sets.

\smallskip
\begin{theorem}\label{k}
Let $A$ be a set of natural numbers
with $\overline{d}(A)=\alpha>0$. Then for every
$\varepsilon>0$ and for every set $B$ 
with $|B|\ge \alpha/\varepsilon$, one has
$$\{x\mid\overline{d}(A\cap(A+x))\ge
\alpha^2-\varepsilon\}\cap\Delta(B)\ne\emptyset\,.$$
\end{theorem}

\begin{proof}
Notice first that the limit superior for the upper asymptotic density 
is attained along intervals of the form $[1,a_n]$; so, by passing
to a subsequence if necessary, we can directly assume that
$\lim_{n\to\infty}n/a_n=\alpha$. 
Without loss of generality, let us assume that 
$B=\{b_1<\ldots<b_\nu\}$ is finite with $\nu\ge\alpha/\varepsilon$. 
For every $n$, apply 
Lemma \ref{lemma} to the finite sets $A_n=\{a_1<\ldots<a_n\}$ 
and $B$ (with $h=1$)
and obtain the existence of an element $x_n\in\Delta(B)$ such that 
$$|A\cap(A+x_n)\cap[1,a_n]|\ \ge\ 
|A\cap(A+x_n)\cap[1,a_n+b_\nu]| - b_\nu\ \ge$$
$$\ge\ |A_n\cap(A_n+x_n)|-b_\nu\ \ge\ 
\frac{n^2}{a_n}\cdot
\frac{1-
\frac{a_n+b_\nu}{n\cdot\nu}}
{1+\frac{b_\nu}{a_n}} - b_\nu\,.$$
Since $\nu$ is fixed, by passing to the limit as $n$ goes to infinity, we get
$$\lim_{n\to\infty}\frac{|A\cap(A+x_n)\cap[1,a_n]|}{a_n}\ \ge$$
$$\ge\ 
\lim_{n\to\infty}\frac{n^2}{a_n^2}\cdot
\frac{1-
\frac{a_n}{n\nu}-\frac{b_\nu}{n\nu}}
{1+\frac{b_\nu}{a_n}} - \frac{b_\nu}{a_n}\ =\ 
\alpha^2\cdot\left(1-\frac{1}{\alpha\cdot\nu}\right)\ \ge\ \alpha^2-\varepsilon\,.$$

Now notice that the sequence $\langle x_n\mid n\in\N\rangle$
takes values in the finite set $\Delta(B)$, and so 
there exists an element $x\in\Delta(B)$ such that
the limit superior is attained along
a subsequence $\{n_k\}$ where $x_{n_k}=x$ for all $k$. 
Such an element $x$ yields the thesis because
$$\overline{d}(A\cap(A+x))\ \ge\ \limsup_{k\to\infty}
\frac{|A\cap(A+x)\cap[1,a_{n_k}]|}{a_{n_k}}\ =$$
$$=\ 
\limsup_{k\to\infty}\frac{|A\cap(A+x_{n_k})\cap[1,a_{n_k}]|}{a_{n_k}}\ 
\ge\ \alpha^2-\varepsilon\,.$$
\end{proof}

\smallskip
As a straight corollary, we obtain
the well-known density version of \emph{Khintchine's Recurrence Theorem}
for sets of integers (see, \emph{e.g.}, \S 5 of \cite{be2}).

\smallskip
\begin{corollary}\label{kk}
Let $A=\{a_n\}$ and $B=\{b_n\}$ be infinite sets of natural numbers.
If $\overline{d}(A)>0$ then
for every $\varepsilon>0$ the following intersection is infinite:
$$\{x\mid \overline{d}(A\cap(A+x))\ge
\overline{d}(A)^2-\varepsilon\}\cap\Delta(B).$$
In consequence:

\begin{enumerate}
\item
All intersections $R_k(A)\cap\Delta(B)$ are infinite\,;

\smallskip
\item
$\limsup_{x\in\Delta(B)}
\frac{\overline{d}(A\cap (A+x))}{\overline{d}(A)^2}\ge 1$\,.
\end{enumerate}
\end{corollary}

\begin{proof}
For every $h$, by applying the previous theorem to
$A$ and $B^h=\{b_{hn}\}$, one gets the existence of an 
element
$x_h=b_{hs}-b_{ht}\in\Delta(B^h)\subseteq\Delta(B)$
with $\overline{d}(A\cap(A+x_h))\ge\overline{d}(A)^2-\varepsilon$.
Notice that $x_h\ge hs-ht\ge h$. This proves
that there are arbitrarily large elements in the intersection
$\{x\mid \overline{d}(A\cap(A+x))\ge
\overline{d}(A)^2-\varepsilon\}\cap\Delta(B)$, as desired.

(1). Every set of positive upper density is infinite, and so, for every $k$,
the set $\{x\mid \overline{d}(A\cap(A+x))>
\overline{d}(A)^2-\varepsilon\}\subseteq R_k(A)$
whenever $0<\varepsilon<\overline{d}(A)^2$.

(2). By what proved above, for every $\varepsilon>0$
there are infinitely many elements $x\in\Delta(B)$
such that $\overline{d}(A+(A+x))\ge\overline{d}(A)^2-\varepsilon$; 
but then 
$\limsup_{x\in\Delta(B)}\overline{d}(A\cap (A+x))\ge
\overline{d}(A)^2-\varepsilon$.
Since $\varepsilon>0$ can be taken arbitrarily small, the thesis follows.
\end{proof}

\smallskip
Further on in this paper, we will show that a similar result as (2) 
can be proved for a large class of zero density sets (see Corollary \ref{main2}).

\medskip
\section{Intersection properties}

We saw in Theorem \ref{thm1} that
$\Delta(A)\cap\Delta(B)\ne\emptyset$ whenever both $A$ and $B$
are asymptotically larger than the set of squares.
We now sharpen that result, and prove a general 
intersection property that also applies when $b_n/n^2$ goes to infinity.

\smallskip
\begin{theorem}\label{maintheorem}
Let $A=\{a_n\}$ and $B=\{b_n\}$ be infinite sets of natural numbers
where $a_n\lll n^2$. Denote by $f(n)=a_n/n$ and by $g(n)=b_n/n$. 

\smallskip
\begin{enumerate}
\item
If there exists a constant $c>1$ such that
$$\liminf_{n\to\infty}\ \frac{g(\lfloor c\cdot f(n)\rfloor)}{n}\ <\ 1-\frac{1}{c}$$
then $R_k(A)\cap\Delta(B)\ne\emptyset$ for all $k$.

\smallskip
\item
If for arbitrarily large constants $c$ one has
$$\liminf_{n\to\infty}\ \frac{g(\lfloor c\cdot f(n)\rfloor)}{n}\ =\ 0$$
then $R_k(A)\cap\Delta(B)$ is infinite for all $k$.

\smallskip
\item
If there exists a constant $\varepsilon>0$ such that
$$\liminf_{n\to\infty}\frac{f(\lfloor\varepsilon\cdot b_n\rfloor)}{n}\ <\ 1$$
then $R_k(A)\cap\Delta(B)\ne\emptyset$ for all $k$.

\smallskip
\item
If there exists a constant $\varepsilon>0$ such that
$$\liminf_{n\to\infty}\frac{f(\lfloor\varepsilon\cdot b_n\rfloor)}{n}\ =\ 0$$
then $R_k(A)\cap\Delta(B)$ is infinite for all $k$.
\end{enumerate}
\end{theorem}

\begin{proof}
In the following, without loss of generality,
we will always assume that $n\lll a_n$. Indeed, 
$n\lll a_n$ fails if and only if the upper asymptotic density $\overline{d}(A)$
is positive, and in this case the four properties above 
are all proved by Corollary \ref{kk}.

\smallskip
(1). Let
$$\liminf_{n\to\infty}\frac{g(\lfloor c\cdot f(n)\rfloor)}{n}\ =\ l\ <\ 1-\frac{1}{c}.$$
For every $n$, let $\tau(n)=\lfloor c\cdot f(n)\rfloor$,
and apply Lemma \ref{lemma} with $h=1$ to the sets
$A_n=\{a_1<\ldots<a_n\}$ and $B_{\tau(n)}=\{b_1<\ldots<b_{\tau(n)}\}$.
We obtain the existence of an element
$x_n\in\Delta(B_{\tau(n)})\subseteq\Delta(B)$ such that:
$$|A\cap(A+x_n)\cap[1,a_n+b_{\tau(n)}]|\ \ge\ |A_n\cap(A_n+x_n)|\ \ge\ 
\frac{n^2}{a_n}\cdot
\frac
{1-\frac{a_n+b_{\tau(n)}} {n\cdot \tau(n)}}
{1+\frac{b_{\tau(n)}}{a_n}}\,.$$
Since we are assuming $n\lll a_n$, we have that
$\lim_{n\to\infty}f(n)=\infty$, and so
$$\lim_{n\to\infty}
\frac{a_n}{n\cdot \tau(n)}\ =\
\lim_{n\to\infty}\frac{f(n)}{\lfloor c\cdot f(n)\rfloor}\ =\ \frac{1}{c}\,.$$
Besides,
$$\liminf_{n\to\infty}
\frac{b_{\tau(n)}}{n\cdot \tau(n)}\ =\
\liminf_{n\to\infty}\frac{g(\lfloor c\cdot f(n)\rfloor)}{n}\ =\ l\,,$$
and
$$\liminf_{n\to\infty}\frac{b_{\tau(n)}}{a_n}\ =\
\liminf_{n\to\infty}\frac{\lfloor c\cdot f(n)\rfloor}{f(n)}\cdot
\frac{g(\lfloor c\cdot f(n)\rfloor)}{n}\ =\ c\cdot l\,.$$
Notice that the two limit inferiors above are attained along the same
subsequence, and so
$$\limsup_{n\to\infty}
\frac
{1-
\frac{a_n+b_{\tau(n)}}{n\cdot \tau(n)}}
{1+\frac{b_{\tau(n)}}{a_n}}\ =\
\frac{1-\left(\frac{1}{c}+l\right)}{1+c\cdot l}\ >\ 0\,.$$
By using the hypothesis $a_n\lll n^2$, \emph{i.e.}
$\lim_{n\to\infty}{n^2}/{a_n}=\infty$,
we can then conclude that
$$\limsup_{n\to\infty}|A\cap(A+x_n)\cap[1,a_n+b_{\tau(n)}]|\ =\ \infty\,.$$
This shows that for every $k$ one finds elements $x_n\in\Delta(B)$
such that $|A\cap(A+x_n)|\ge k$, and hence 
$R_k(A)\cap\Delta(B)\ne\emptyset$.\footnote
{~We remark that the map $n\mapsto x_n$ may not be 1-1,
and so the above argument \emph{does not} prove that
$R_k(A)\cap\Delta(B)$ contains infinitely many elements.}

\smallskip
(2). Fix $h>1$. For every $n$, let
$\tau(n)=\lfloor 2h\cdot f(n)\rfloor$, and apply 
Lemma \ref{lemma} to the sets $A_n$ and $B_{\tau(n)}$ so as
to get the existence of 
an element $x_n\in\Delta(B_{\tau(n)})\subseteq\Delta(B)$ such that $x_n\ge h$ and
$$|A\cap(A+x_n)\cap[1,a_n+b_{\tau(n)}|\ \ge\
\frac{n^2}{a_n}\cdot
\frac
{1-\frac{(a_n+b_{\tau(n)})(2h-1)} {n\cdot \tau(n)}}
{1+\frac{b_{\tau(n)}}{a_n}}.$$
Now use the same arguments as in the proof of 
the previous property (1). Since in our case $c=2h$ and $l=0$, we obtain that
$$\limsup_{n\to\infty}\,\frac
{1-\frac{(a_n+b_{\tau(n)})(2h-1)} {n\cdot \tau(n)}}
{1+\frac{b_{\tau(n)}}{a_n}}\ =\ 
\frac
{1-\left(\frac{1}{c}+l\right)(2h-1)}
{1+c\cdot l}\ =\ 1-\frac{2h-1}{2h}\ >\ 0.$$
By the hypothesis $a_n\lll n^2$, we conclude that
$$\limsup_{n\to\infty}\,|A\cap(A+x_n)\cap[1,a_n+b_{\tau(n)}]|\ =\ \infty.$$
So, for every $k$, there exist elements $x_n\in\Delta(B_n)\subseteq\Delta(B)$ 
such that $x_n\ge h$
and $|A\cap(A+x_n)|\ge k$. Since $h$ is arbitrary, this shows
that the intersection $R_k(A)\cap\Delta(B)$ is infinite, as desired.

\smallskip
(3). The proof is entirely similar to the proof of (1), by
applying Lemma \ref{lemma} to the sets
$A_{\sigma(n)}$ and $B_n$ where $\sigma(n)=\lfloor\varepsilon\cdot b_n\rfloor$.
Indeed, notice that
$$\liminf_{n\to\infty}\frac{a_{\sigma(n)}}{\sigma(n)\cdot n}\ =\ 
\liminf_{n\to\infty}\frac{f(\lfloor\varepsilon\cdot b_n\rfloor)}{n}\ =\ l\ <\ 1.$$
Besides,
$$\lim_{n\to\infty}\frac{b_n}{\sigma(n)}\ =\ 
\lim_{n\to\infty}\frac{b_n}{\lfloor\varepsilon\cdot b_n\rfloor}\ =\ 
\frac{1}{\varepsilon}\ <\ \infty\,,$$
and so
$$\lim_{n\to\infty}\frac{b_n}{\sigma(n)\cdot n}\ =\ 0\quad
\text{and}\quad
\lim_{n\to\infty}\frac{b_n}{a_{\sigma(n)}}\ =\ 
\lim_{n\to\infty}\frac{b_n}{\sigma(n)}\cdot
\frac{\sigma(n)}{a_{\sigma(n)}}\ =\ 0.$$
Thus we have the existence of elements $x_n\in\Delta(B)$ such that
$$\limsup_{n\to\infty}|A\cap(A+x_n)\cap[1,a_{\sigma(n)}+b_n]|\ \ge\ 
\limsup_{n\to\infty}
\frac{\sigma(n)^2}{a_{\sigma(n)}}\cdot
\frac{1-
\frac{a_{\sigma(n)}+b_n}{\sigma(n)\cdot n}}
{1+\frac{b_n}{a_{\sigma(n)}}}\ =\ \infty\,,$$
and the thesis follows.

\smallskip
(4).  For fixed $h>1$, we proceed as in (3) 
and obtain the existence of elements $x_n\in\Delta(B_n)\subseteq\Delta(B)$ 
with $x_n\ge h$ and such that
$$\limsup_{n\to\infty}|A\cap(A+x_n)\cap[1,a_{\sigma(n)}+b_n]|\ \ge\ 
\limsup_{n\to\infty}
\frac{\sigma(n)^2}{a_{\sigma(n)}}\cdot
\frac{1-
\frac{(a_{\sigma(n)}+b_n)(2h-1)}{\sigma(n)\cdot n}}
{1+\frac{b_n}{a_{\sigma(n)}}}\,.$$
As we are assuming $l=0$, the above limit superior is infinite.
Finally, since $h$ can be taken arbitrarily large, the thesis follows.
\end{proof}

\smallskip
\begin{remark}
Under the (mild) hypothesis that $g(n)$ be non-decreasing, one can
prove (3) and (4) as consequences of (1) and (2),
which are therefore basically stronger properties.
Indeed, given $\varepsilon>0$,
let us assume that $\tau(n)=f(\lfloor\varepsilon\cdot b_n\rfloor)/n$
satisfies the condition $\liminf_{n\to\infty}\tau(n)=l<1$.
Then for every constant $c$ such that $c\cdot l<1$, we have
that $c\cdot\tau(n)\cdot n<n$ for infinitely many $n$, and so
$$\liminf_{n\to\infty}\frac{g(\lfloor c\cdot f(n)\rfloor)}{n}\ \le\ 
\liminf_{n\to\infty}
\frac{g(\lfloor c\cdot f(\lfloor\varepsilon\cdot b_n\rfloor)\rfloor)}
{\lfloor\varepsilon\cdot b_n\rfloor}
 =$$
 $$=\ \liminf_{n\to\infty}\frac{g(\lfloor c\cdot \tau(n)\cdot n\rfloor)}
 {\lfloor\varepsilon\cdot n\cdot g(n)\rfloor}\ =\ 
 \frac{1}{\varepsilon}\cdot\liminf_{n\to\infty}
\frac{g(\lfloor c\cdot \tau(n)\cdot n\rfloor)}{n\cdot g(n)}\ \le$$
$$\le\ \frac{1}{\varepsilon}\cdot\liminf_{n\to\infty}\frac{g(n)}{n\cdot g(n)}\ =\ 0\,.$$
Notice that, since $l<1$, we can pick constants $c>1$
such that $c\cdot l<1$, and this completes the proof of $(1)\Rightarrow(3)$.
Besides, if $l=0$, every constant $c>1$ trivially satisfies $c\cdot l<1$,
and also $(2)\Rightarrow(4)$ follows.
\end{remark}

\smallskip
As a consequence of the previous theorem,
one can isolate a large class of sets 
$B$ such that $R_k(A)\cap\Delta(B)\ne\emptyset$,
in terms of their density relative to $A$.
 
\smallskip
\begin{corollary}\label{cor1}
Let $A=\{a_n=n\cdot f(n)\}$ be an infinite set of natural numbers
where $f:\R^+\to\R^+$ is an increasing 
unbounded function, and assume that
the infinite set of natural numbers $B=\{b_n\}$ is such that
$$\lim_{n\to\infty}\,\frac{b_n/n}{f^{-1}(\varepsilon\cdot n)}\ =\ 0
\quad\text{for all }\varepsilon>0\,.$$
Then intersections $R_k(A)\cap\Delta(B)$ are infinite for all $k$.
\end{corollary}

\begin{proof}
Fix $c>1$, and let $\tau(n)=\lfloor c\cdot f(n)\rfloor$ and $\varepsilon=1/c$.
Then $f^{-1}(\varepsilon\cdot\tau(n))\le n$ and
$$0\ \le\ \lim_{n\to\infty}\frac{g(\lfloor c\cdot f(n)\rfloor)}{n}\ \le\ 
\lim_{n\to\infty}\frac{g(\tau(n))}{f^{-1}(\varepsilon\cdot\tau(n))}\ =\ 
\lim_{n\to\infty}\frac{{b_{\tau(n)}}/{\tau(n)}}{f^{-1}(\varepsilon\cdot\tau(n))}\ =\ 0\,.$$
Thus (2) of the previous Theorem applies, and we get the thesis.
\end{proof}

\smallskip
When $\varepsilon=1$, items (3) and (4) in Theorem \ref{maintheorem}
have the advantage that can be reformulated in the following simpler form:

\smallskip
\begin{corollary}\label{abn}
Let $A=\{a_n\}$ and $B=\{b_n\}$ be infinite sets of natural numbers
where $a_n\lll n^2$, and let 
$$\liminf_{n\to\infty}\,\frac{a_{b_n}}{n\cdot b_n}\ =\ l\,.$$
If $l<1$ then $R_k(A)\cap\Delta(B)\ne\emptyset$ for all $k$; 
and if $l=0$ then $R_k(A)\cap\Delta(B)$ is infinite for all $k$.
\end{corollary}

\smallskip
A consequence that is easily applied in several examples is the following:

\smallskip
\begin{corollary}\label{main1}
Given a function $\vartheta:\N\to\R^+$ and infinite sets of natural numbers
$A=\{a_n\}$ and $B=\{b_n\}$, denote by:
$$\liminf_{n\to\infty}\frac{a_n}{n\cdot\vartheta(n)}=\underline{\ell}\,;\quad
\limsup_{n\to\infty}\frac{a_n}{n\cdot\vartheta(n)}=\overline{\ell}\,;$$
$$\liminf_{n\to\infty}\frac{\vartheta(b_n)}{n}=\underline{\ell}\,'\,;\quad
\limsup_{n\to\infty}\frac{\vartheta(b_n)}{n}=\overline{\ell}\,'\,.$$
If $\underline{\ell}\cdot \overline{\ell}\,'<1$ or 
$\overline{\ell}\cdot \underline{\ell}\,'<1$ then $R_k(A)\cap\Delta(B)\ne\emptyset$ for all $k$;
and if $\underline{\ell}\cdot\overline{\ell}\,'=0$
or $\underline{\ell}\cdot\overline{\ell}\,'=0$
then $R_k(A)\cap\Delta(B)$ is infinite for all $k$.\footnote
{~By writing $\underline{\ell}\cdot \overline{\ell}\,'<1$ or
$\underline{\ell}\cdot \overline{\ell}\,'=0$, it is implicitly
assumed that both $\underline{\ell}$ and $\overline{\ell}\,'$ are finite; and
similarly in the other cases.}
\end{corollary}

\begin{proof}
It is a direct application of Corollary \ref{abn}. Indeed, if $\underline{\ell}$
and $\overline{\ell}\,'$ are finite, then 
$$\liminf_{n\to\infty}\frac{a_{b_n}}{n\cdot b_n}\ \le\
\liminf_{n\to\infty}\frac{a_{b_n}}{b_n\cdot\vartheta(b_n)}\cdot
\limsup_{n\to\infty}\frac{\vartheta(b_n)}{n}\ \le\ \underline{\ell}\cdot\overline{\ell}\,'\,;$$
and if $\overline{\ell}$ and $\underline{\ell}\,'$ are finite, then 
$$\liminf_{n\to\infty}\frac{a_{b_n}}{n\cdot b_n}\ \le\
\limsup_{n\to\infty}\frac{a_{b_n}}{b_n\cdot\vartheta(b_n)}\cdot
\liminf_{n\to\infty}\frac{\vartheta(b_n)}{n}\ \le\ \overline{\ell}\cdot\underline{\ell}\,'\,.$$
\end{proof}

\smallskip
As witnessed by the results proved above, 
if $A$ has zero density but still it is ``large" enough,
then its set of distances intersect sets of distances of
really ``sparse'' sets $B$.
We give below two examples to illustrate this phenomenon.

\smallskip
\begin{example}
Let $P=\{p_n\}$ be the set of prime numbers, and
let $B=\{2^n\}$ be the set of powers of $2$. By the 
\emph{Prime Number Theorem}, 
$$\lim_{n\to\infty}\frac{p_n}{n\cdot\log n}\ =\ 1.$$
Since $(\log 2^n)/n=\log 2<1$, by the previous corollary
we can conclude 
that for every $k$, there exist numbers of the form
$2^m-2^n$ which are the distance of at least $k$-many pairs
of primes. Actually, there exist infinitely many 
such numbers, since the function $(n,m)\mapsto 2^m-2^n$ is 1-1;
indeed, first pick $2^{n_1}-2^{m_1}\in R_k(P)\cap\Delta(B)$, then 
consider $B^{(1)}=B\setminus\{2^{n_1},2^{n_2}\}$ and pick
$2^{n_2}-2^{m_2}\in R_k(P)\cap\Delta(B^{(1)})$, and so forth.
\end{example}

\smallskip
\begin{example}\label{ex1}
Let $A=\{a_n\}$ and $B=\{b_n\}$ be infinite sets of natural numbers
such that
$$\sum_{n=1}^\infty\frac{1}{a_n}\ =\ \infty\quad\text{and}\quad
\log b_n\ \lll\ n^{1-\varepsilon}\ \
\text{for some}\ \varepsilon>0\,.$$
Then $R_k(A)\cap\Delta(B)$ is infinite for all $k$.
\end{example}

\begin{proof}
If we let $\vartheta(n)=(\log n)^{\frac{1}{1-\varepsilon}}$,
the hypotheses imply that
$$\liminf_{n\to\infty}\frac{a_n}{n\cdot\vartheta(n)}=0\quad\text{and}\quad
\limsup_{n\to\infty}\frac{\vartheta(b_n)}{n}=
\left(\limsup_{n\to\infty}\frac{\log b_n}{n^{1-\varepsilon}}\right)^{\frac{1}{1-\varepsilon}}=0\,,$$
and the desired intersection property
follows by Corollary \ref{main1}.
\end{proof}

\smallskip
\emph{E.g.}, if $A=\{a_n\}$ is such that $\sum_{n=1}^\infty\frac{1}{a_n}=\infty$,
then for every exponent $\alpha<1$ and for every $k$, there exist infinitely many
numbers of the form $\lfloor 10^{n^\alpha}\rfloor - \lfloor 10^{m^\alpha}\rfloor$,
everyone of which is the distance of at least
$k$-many different pairs of elements of $A$.

\smallskip
Let us now focus on powers of $n$. 

\smallskip
\begin{theorem}\label{powersofn}
Let $A=\{a_n\}$ and $B=\{b_n\}$ be infinite sets of natural numbers
such that, for all sufficiently large $n$, 
$$a_n\le K\cdot n^{1+\alpha}\quad\text{and}\quad b_n\le M\cdot n^{1+\beta}\,.$$

\begin{enumerate}
\item
If $\alpha<1$ and $\beta<1/\alpha$ then $R_k(A)\cap\Delta(B)$ is infinite for all $k$.

\smallskip
\item
If $\alpha<1$ and $\beta=1/\alpha$ 
then $R_k(A)\cap\Delta(B)\ne\emptyset$ for all $k$ whenever  
$K^\beta M<\frac{\alpha}{\left(1+\alpha\right)^{\beta+1}}$.

\item
If $\alpha=\beta=1$ then $\Delta(A)\cap\Delta(B)\ne\emptyset$ 
whenever $KM<\frac{1}{4}$.
\end{enumerate}
\end{theorem}

\begin{proof}
Notice first that, without loss of generality, 
we can assume $n\lll a_n$, and hence $\alpha>0$. 
Indeed, otherwise $\overline{d}(A)>0$, and 
the thesis is proved by Corollary \ref{kk}.

\smallskip
(1).  The thesis follows from (2) of 
Theorem \ref{maintheorem} since $a_n\lll n^2$ and
for every constant $c>1$ one has that
$$\liminf_{n\to\infty}\ \frac{g(\lfloor c\cdot f(n)\rfloor)}{n}\ \le\ 
\lim_{n\to\infty}\frac{M\cdot\left(c\cdot K\cdot n^\alpha\right)^\beta}{n}\ =\ 
\lim_{n\to\infty}M\cdot c^\beta\cdot K^\beta\cdot\frac{n^{\alpha\beta}}{n}\ =\  0\,.$$

\smallskip
(2). We use (1) of Theorem \ref{maintheorem}.
Given a constant $c>1$, under our hypotheses one has that
$$\liminf_{n\to\infty}\ \frac{g(\lfloor c\cdot f(n)\rfloor)}{n}\ \le\ 
M\cdot c^\beta\cdot K^\beta.$$
Now,
$$M\cdot c^\beta\cdot K^\beta\ <\ 1-\frac{1}{c}\ \Longleftrightarrow\ 
M\cdot K^\beta\ <\ \frac{c-1}{c^{\beta+1}}\,,$$
and the greatest possible value of the last expression is attained when $c=1+\alpha$,
namely $\frac{\alpha}{\left(1+\alpha\right)^{\beta+1}}$,
as one can directly verify. 

\smallskip
(3). Fix a constant $c>0$. For every given $n$, let
$N=n$ and $\nu=\tau(n)=\lfloor c\cdot\sqrt{K/M}\cdot n\rfloor$. 
By Lemma \ref{lemma}, there exists an element $x_n\in\Delta(B)$ 
such that
$$|A\cap(A+x_n)\cap[1,a_n+b_{\tau(n)}]|\ \ge 
\frac{n^2}{a_n+b_{\tau(n)}}-\frac{n}{\tau(n)}\ \ge$$
$$\ge\ 
\frac{n^2}{K n^2+M c^2\cdot\frac{K}{M}\cdot n^2}-
\frac{n}{\lfloor c\cdot\sqrt{\frac{K}{M}}\cdot n\rfloor}\ =\ 
\frac{1}{K}\cdot\left(\frac{1}{1+c^2}-
\frac{\sqrt{KM}}{c}\cdot\psi(n)\right)$$
where 
$\psi(n)=\frac{c\cdot \sqrt{K/M}\cdot n}{\lfloor c\cdot \sqrt{K/M}\cdot n\rfloor}\longrightarrow 1$ as $n\to\infty$.
So, the last quantity above is positive for all sufficiently large $n$
if and only if $\sqrt{KM}<\frac{c}{1+c^2}$. Now, it is easily checked that
the greatest possible value of the latter expression is $1/2$,
which is attained when $c=1$. This means that if $KM<1/4$ then there
exist elements $x_n\in\Delta(A)\cap \Delta(B)$, \emph{i.e.} the thesis.
\end{proof}

\smallskip
\begin{example}\label{ex2}
Let $A=\{\lfloor K\cdot n\sqrt{n}\rfloor\}$ and $B=\{n^3\}$.
If $K^2\cdot M<4/27$ then $R_k(A)\cap\Delta(B)\ne\emptyset$ for all $k$.
Indeed, we can apply (2) of the theorem above, where $1/\alpha=\beta=2$.
\end{example}

\medskip
\section{A variant of Khintchine's Theorem}

In this final section we exploit further consequences of Lemma \ref{lemma}
and prove a result for a class of zero density sets
that resembles \emph{Khintchine's Recurrence Theorem}.

\smallskip
Let us first introduce some notation.
For sets $A\subseteq\N$, we write $\dd(A)_n$ to denote
the relative density of $A$ on the interval $[1,n]$, \emph{i.e.}
$$\dd(A)_n\ =\ \frac{|A\cap [1,n]|}{n}\,.$$
As already pointed out, the limit superior given
by the upper asymptotic density is
attained along intervals of the form $[1,a_n]$; so one has
$$\overline{d}(A)\ =\ \limsup_{n\to\infty}\dd(A)_{a_n}\ =\ 
\limsup_{n\to\infty}\frac{n}{a_n}.$$

\smallskip
\begin{theorem}\label{abn2}
Let $A=\{a_n\}$ and $B=\{b_n\}$ be infinite sets of natural numbers,
and assume that
$$\liminf_{n\to\infty}\ \frac{a_{n\cdot b_n}}{n^2\cdot b_n}\ =\ l\ <\ \frac{1}{2}\,.$$
Then there exists a sequence $\langle x_n\mid n\in\N\rangle$ of elements of 
$\Delta(B)$ such that
$$\limsup_{n\to\infty}\,
\left(\frac{\dd(A\cap(A+x_n))_{a_n}}
{(\dd(A)_{a_n})^2}\right)\ \ge\ 1-2l\ >\ 0\,.$$
\end{theorem}

\begin{proof}
For every $n$, let $\sigma(n)=n\cdot b_n$,
and apply Lemma \ref{lemma} with $h=1$ to the sets
$A_{\sigma(n)}=\{a_1<\ldots<a_{\sigma(n)}\}$ and 
$B_n=\{b_1<\ldots<b_n\}$.
We obtain the existence of an element
$x_n\in\Delta(B_n)\subseteq\Delta(B)$ such that:
$$|A\cap(A+x_n)\cap[1,a_{\sigma(n)}]|\ \ge\ 
|A\cap(A+x_n)\cap[1,a_{\sigma(n)}+b_n]|-b_n\ \ge$$
$$\ge\ |A_{\sigma(n)}\cap(A_{\sigma(n)}+x_n)|-b_n\ \ge\ 
\frac{\sigma(n)^2}{a_{\sigma(n)}}\cdot
\frac
{1-\frac{a_{\sigma(n)}+b_n} {\sigma(n)\cdot n}}
{1+\frac{b_n}{a_{\sigma(n)}}}-b_n\,.$$
By combining, one gets
$$\frac{\dd(A\cap(A+x_n))_{a_{\sigma(n)}}}{(\dd(A)_{a_{\sigma(n)}})^2}\ =\ 
\frac{|A\cap(A+x_n)\cap[1,a_{\sigma(n)}]|}
{\frac{\sigma(n)^2}{a_{\sigma(n)}}}\ \ge$$
$$\ge\  
\frac
{1-\frac{a_{\sigma(n)}+b_n} {\sigma(n)\cdot n}}
{1+\frac{b_n}{a_{\sigma(n)}}} - 
\frac{a_{\sigma(n)}\cdot b_n}{\sigma(n)^2}\,.$$

Now notice that:

\smallskip
\begin{itemize}
\item
$\liminf_{n\to\infty}\frac{a_{\sigma(n)}}{\sigma(n)\cdot n}=
\liminf_{n\to\infty}\frac{a_{n\cdot b_n}}{n^2\cdot b_n}=l$\,;

\smallskip
\item
$\lim_{n\to\infty}\frac{b_n}{\sigma(n)\cdot n}=
\lim_{n\to\infty}\frac{1}{n^2}=0$\,;

\smallskip
\item
$\lim_{n\to\infty}\frac{b_n}{a_{\sigma(n)}}=
\lim_{n\to\infty}\frac{n\cdot b_n}{a_{n\cdot b_n}}\cdot\frac{1}{n}\le
\lim_{n\to\infty}\frac{1}{n}=0$\,;

\item
$\liminf_{n\to\infty}\frac{a_{\sigma(n)}\cdot b_n}{\sigma(n)^2}=
\liminf\frac{a_{n\cdot b_n}}{n^2\cdot b_n}=l$\,.
\end{itemize}

\smallskip
By considering the inequalities proved above, and by passing
to the limit superiors as $n$ goes to infinity, we finally get:
$$\limsup_{n\to\infty}
\left(\frac{\dd(A\cap(A+x_n))_{a_n}}
{(\dd(A)_{a_n})^2}\right)\ \ge\ 
\limsup_{n\to\infty}\left(\frac{\dd(A\cap(A+x_n))_{a_{\sigma(n)}}}{(\dd(A)_{a_{\sigma(n)}})^2}\right)\ \ge$$
$$\ge\ \limsup_{n\to\infty}
\left(\frac
{1-\frac{a_{\sigma(n)}}{\sigma(n)\cdot n}-\frac{b_n}{\sigma(n)\cdot n}}
{1+\frac{b_n}{a_{\sigma(n)}}} - 
\frac{a_{\sigma(n)}\cdot b_n}{\sigma(n)^2}\right)\ =\ 
1-2l\ >\ 0\,.$$
\end{proof}

\smallskip
\begin{corollary}
Let $A=\{a_n\}$ be an infinite set of natural numbers. 
If $a_n\lll n^{3/2}$ then there exists
a sequence of shifts $\langle x_n\mid n\in\N\rangle$ such that
$$\limsup_{n\to\infty}\,
\left(\frac
{\frac{|A\cap(A+x_n)\cap[1,n]|}{n}}
{\left(\frac{|A\cap[1,n]|}{n}\right)^2}\right)\ \ge\ 1\,.$$
\end{corollary}

\begin{proof}
Let $B=\N$. Then the previous theorem applies where $l=0$,
and the thesis easily follows.
\end{proof}

\smallskip
Similarly as Corollary \ref{main1} is derived
from Theorem \ref{abn}, one proves the following
property as a straight consequence of Theorem \ref{abn2}.

\smallskip
\begin{corollary}\label{main2}
Assume that,  for a suitable $\vartheta:\N\to\R^+$, the infinite
sets of natural numbers
$A=\{a_n\}$ and $B=\{b_n\}$ satisfy 
$$\limsup_{n\to\infty}\frac{a_n}{n\cdot\vartheta(n)}\ =\ l_1\ <\ \infty\quad\text{and}
\quad \liminf_{n\to\infty}\frac{\vartheta(n\cdot b_n)}{n}\ =\ l_2\ <\ \infty$$
where $l_1\cdot l_2<1/2$. 
Then there exists a sequence $\langle x_n\mid n\in\N\rangle$
of elements of $\Delta(B)$ such that
$$\limsup_{n\to\infty}\,
\left(\frac{\dd(A\cap(A+x_n))_{a_n}}
{(\dd(A)_{a_n})^2}\right)\ \ge\ 1-2\,l_1 l_2\ >\ 0\,.$$
\end{corollary}

\begin{proof}
Theorem \ref{abn2} applies, since
$$\liminf_{n\to\infty}\frac{a_{n\cdot b_n}}{n^2\cdot b_n}\ \le\
\limsup_{n\to\infty}\frac{a_{n\cdot b_n}}{n\cdot b_n\cdot\vartheta(n\cdot b_n)}\cdot
\liminf_{n\to\infty}\frac{\vartheta(n\cdot b_n)}{n}\ \le\ l_1\,l_2\ <\ \frac{1}{2}\,.$$
\end{proof}

\smallskip
To illustrate the use of the above corollary, let us see a property
that holds for all sets $A=\{a_n\}$ having the same asymptotic size
as the set of primes.

\smallskip
\begin{proposition}
Let $A=\{a_n\}$ and $B=\{b_n\}$ be infinite sets of natural numbers
such that
$$\lim_{n\to\infty}\frac{a_n}{n \log n} =\ 1\quad\text{and}\quad
\liminf_{n\to\infty}\frac{\log b_n}{n}=0.$$
Then for every $\varepsilon>0$ there exist
infinitely many $n$ and elements $x_n\in\Delta(B)$ such that
$$|A\cap(A+x_n)\cap[1,a_n]|\,\ge\,\frac{n}{\log n}\cdot(1-\varepsilon)\,.$$
\end{proposition}

\begin{proof}
Let $\vartheta(n)=\log n$. By the hypotheses,
$$\lim_{n\to\infty}\frac{a_n}{n\cdot\vartheta(n)}\ =\ 1\quad
\text{and}\quad
\lim_{n\to\infty}\frac{\vartheta(n\cdot b_n)}{n}\ =\ 
\lim_{n\to\infty}\frac{\log n+\log b_n}{n}\ =\ 0\,.$$
So, the previous corollary applies, and we get the existence
of elements $x_n\in\Delta(B)$ such that
$$\limsup_{n\to\infty}\,
\left(\frac{\dd(A\cap(A+x_n))_{a_n}}
{(\dd(A)_{a_n})^2}\right)\ \ge\ 1\,.$$
Now notice that
$$\frac{\dd(A\cap(A+x_n))_{a_n}}
{(\dd(A)_{a_n})^2}\ =\ 
|A\cap(A+x_n)\cap[1,a_n]|\cdot\frac{a_n}{n^2}\,.$$
So, for every $\delta>0$, there exist infinitely many $n$ that satisfy
$$|A\cap(A+x_n)\cap[1,a_n]|\cdot\frac{a_n}{n^2}\ \ge\ 1-\delta\,.$$
By our hypothesis on $\{a_n\}$, we
know that $\frac{n\cdot\log n}{a_n}\ge 1-\delta$ for
all sufficiently large $n$, and so we can conclude that there exist infinitely 
many $n$ and elements $x_n\in\Delta(B)$ such that:
$$|A\cap(A+x_n)\cap[1,a_n]|\ \ge\ \frac{n^2}{a_n}\cdot(1-\delta)\ =\ 
\frac{n}{\log n}\cdot\frac{n\,\log n}{a_n}\cdot(1-\delta)\ \ge\ 
\frac{n}{\log n}\cdot(1-\delta)^2\,.$$
The proof is completed by choosing $\delta$ in such a way 
that $(1-\delta)^2\ge 1-\varepsilon$.
\end{proof}

\smallskip
\begin{example}\label{ex3}
Let $P=\{p_n\}$ be the set of prime numbers.
Then, for any given $\varepsilon>0$, there exist arbitrarily large $n$
such that one finds  ``nearly" $(n/\!\log n)$-many
pairs of primes $p,p'\le p_n$ which have a common distance $p-p'=d$.
Moreover, such a distance $d$ can be taken to belong to
any prescribed set of distances $\Delta(B)$, provided $B=\{b_n\}$ is not
too sparse in the precise sense that $\log b_n\lll n$
(\emph{e.g.}, one can take 
$b_n=\lfloor 10^{\frac{n}{\log n}}\rfloor$).
\end{example}

\end{document}